\documentclass[12pt, reqno]{amsart}
\usepackage{geometry}
\usepackage{amssymb,amsmath,amsthm}
\usepackage{mathtools}
\usepackage[english]{babel}
\usepackage{hyperref}
\usepackage{color}

\usepackage{tikz}
\usepackage{tikz-cd}
\usepackage{comment}
\usetikzlibrary{arrows}
\usetikzlibrary{calc,patterns}
\usetikzlibrary{decorations.markings}
\usetikzlibrary{patterns,decorations.pathreplacing}

\usepackage{mathptmx}
\usepackage{bm}

\newcommand{\rank}{\operatorname {rank}}

\newcommand{\mf}[1]{\mathfrak #1}
\newcommand{\rme}{\mathrm{e}}
\newcommand{\m}{\mathfrak{m}}
\newcommand{\frm}{\mathfrak{m}}
\newcommand{\frp}{\mathfrak{p}}
\newcommand{\bbZ}{\mathbb{Z}}

\newcommand{\bbR}{\mathbb{R}}
\newcommand{\bbP}{\mathbb{P}}
\newcommand{\type}{\mathrm{type}}
\newcommand{\emb}{\mathrm{emb}}

\newcommand{\sfM}{\mathsf{M}} 
\newcommand{\sfN}{\mathsf{N}} 

\DeclareMathOperator{\ehk}{e_{HK}}

\DeclareMathOperator{\fsig}{s}

\DeclareMathOperator{\Hom}{Hom}
\DeclareMathOperator{\Min}{Min}

\DeclareMathOperator{\Assh}{Assh}
\DeclareMathOperator{\Index}{index}

\DeclareMathOperator{\Cl}{Cl}

\DeclareMathOperator{\GL}{GL}
\renewcommand{\hat}{\widehat}

\newcommand{\PC}[1]{\mathcal{P}^{\vee} (\{ #1 \}) }

\newtheorem{thm}{Theorem}
\newtheorem{lemma}[thm]{Lemma}
\newtheorem{prop}[thm]{Proposition}
\newtheorem{cor}[thm]{Corollary}
\theoremstyle{definition}
\newtheorem{defn}[thm]{Definition}
\newtheorem{conj}[thm]{Conjecture}

\theoremstyle{remark}
\newtheorem{remark}[thm]{Remark}
\newtheorem{exam}[thm]{Example}
\newtheorem{quest}[thm]{Question}

\newtheorem{obs}[thm]{Observation} 

\numberwithin{thm}{section}
\numberwithin{equation}{section}

\begin{document}
\title{Lower bounds on Hilbert--Kunz multiplicities and maximal $F$-signatures}
\author[Jack Jeffries]{Jack Jeffries}
\address[Jack Jeffries]{Department of Mathematics, University of Nebraska-Lincoln, Lincoln, NE, 68588, USA}
\email{jack.jeffries@unl.edu}

\author[Yusuke Nakajima]{Yusuke Nakajima}
\address[Yusuke Nakajima]{Department of Mathematics, Kyoto Sangyo University, Motoyama, Kamigamo, Kita-Ku, Kyoto, Japan, 603-8555} 
\email{ynakaji@cc.kyoto-su.ac.jp}

\author[Ilya Smirnov]{Ilya Smirnov}
\address[Ilya Smirnov]{BCAM -- Basque Center for Applied Mathematics, Mazarredo 14, 48009 Bilbao, Basque Country -- Spain
\quad and \quad
IKERBASQUE, Basque Foundation for Science, Plaza Euskadi 5, 48009 Bilbao, Basque Country -- Spain}
\email{ismirnov@bcamath.org}

\author[Kei-ichi Watanabe]{Kei-ichi Watanabe}
\address[Kei-ichi Watanabe]{Department of Mathematics, College of Humanities and Sciences, Nihon University, 3-25-40 Sakurajosui, Setagaya-Ku, Tokyo 156-8550, Japan and 
Organization for the Strategic Coordination of Research and Intellectual Properties, Meiji University}
\email{watnbkei@gmail.com}

\author[Ken-ichi Yoshida]{Ken-ichi Yoshida}
\address[Ken-ichi Yoshida]{Department of Mathematics, College of Humanities and Sciences, Nihon University, 3-25-40 Sakurajosui, Setagaya-Ku, Tokyo 156-8550, Japan}
\email{yoshida@math.chs.nihon-u.ac.jp}

\begin{abstract} 
Hilbert--Kunz multiplicity and F-signature are numerical invariants of commutative rings in positive characteristic
that measure severity of singularities: for a regular ring both invariants are equal to one and the converse holds under mild assumptions. 
A natural question is for what singular rings these invariants are closest to one. 
For Hilbert--Kunz multiplicity this question was first considered by the last two authors and attracted significant attention.  
In this paper, we study this question, i.e., an upper bound, for F-signature 
and revisit lower bounds on Hilbert--Kunz multiplicity. 
\end{abstract}

\maketitle


\section{Introduction}

\subsection{Background}
A ring of positive characteristic has a wealth of objects arising from the Frobenius endomorphism. 
The focus of this paper are two numerical invariants: Hilbert--Kunz multiplicity and F-signature. 
For simplicity, let us assume that $A$ is a local domain such that $A^{1/p}$ is a finitely generated $A$-module.
The Hilbert--Kunz multiplicity of $A$ (\cite{Kunz76, Mon83}) is defined as 
\[
\ehk(A)\coloneqq \lim_{e \to \infty} \dfrac{\mu_A(A^{1/p^e})}{\rank(A^{1/p^e})},
\] 
where $\mu_A$ denotes the minimal number of generators, 
and the F-signature of $A$ (\cite{HunLeu, Tuc12}) is 
\[
\fsig(A) \coloneqq \displaystyle{\lim_{e \to \infty}} \dfrac{\max\{n \mid A^{1/p^e} \cong A^{\oplus n} \oplus M\}}{\rank(A^{1/p^e})}, 
\]
where $M$ is a finitely generated $A$-module without free direct summands. 

A fundamental result of Kunz (\cite{Kunz69}) asserts that $A^{1/p^e}$ is free if and only if $A$ is regular.
It follows that $\ehk(A) \ge 1$ and $1 \ge \fsig(A) \ge 0$, and under a mild condition the value is $1$ if and only if $A$ is regular (\cite{Kunz69, WataYos00, HunLeu}). 
Furthermore, positivity of F-signature characterizes the class of strongly F-regular rings \cite{AberLeu}, 
a fundamental class of mild singularities that first appeared in the tight closure theory \cite{HochHun94}.
A related result of Blickle--Enescu \cite{BlicEnes04} shows that small Hilbert--Kunz multiplicity 
also forces the ring to be strongly F-regular.
\par
A natural question is how close can
the Hilbert--Kunz multiplicity of a singularity be to $1$? And a natural guess is that the simplest double point singularity $k[[x_1, \ldots, x_d]]/(x_1^2 + \cdots + x_d^2)$ should have the smallest Hilbert--Kunz multiplicity (see, Conjecture~\ref{LBconj} for details). 
By \cite{WataYos01a, WataYos05, AberEnes13} this is now a theorem in dimension at most $6$.
\par
In this paper, we extend this investigation by asking to find further bounds on Hilbert--Kunz multiplicity of mild singularities 
and considering the analogous question for F-signature.  
For instance, in dimension~$2$, most non-regular F-regular local rings are quotient singularities, 
in which case we have that $\fsig(A)=1/|G| \le \frac{1}{2}$, where $A=k[[x,y]]^{G}$ and $G$ is a finite subgroup of $\GL_{2}(k)$. 
It seems that a similar question has no answer even in  dimension $3$.  

\par \vspace{2mm} \par \noindent
\begin{quest}\label{que_maxFsig}
Let $A$ be a strongly F-regular local domain of dimension $d \ge 3$ 
which is not regular.  
Then what is the upper bound on $\fsig(A)$?
\end{quest}

We give a partial answer to the question above, 
and pose a conjecture; see Conjecture \ref{FsigConj}. 
Let us explain  the organization of the paper. 

\subsection{Structure of the paper and main results}
In Section~\ref{sec_preliminaries}, we recall several definitions 
(Hilbert--Kunz multiplicity, F-regularity, FFRT, F-signature 
and so on) and pose two conjectures. 
In Section~\ref{section_lowerHK}, we give a lower bound on Hilbert--Kunz multiplicities. 
Namely, we prove the following theorem and its refinement in the $3$-dimensional case 
(see Theorem \ref{3dimOp}). 

\begin{thm}[see Theorem~\ref{Higher}]
Let $(A,\m)$ be a formally unmixed local ring 
of characteristic $p>0$. 
If ${d=\dim A \ge 3}$, then for every $\m$-primary ideal $I$ we have 
\[
\ehk(I) > \dfrac{\rme(I)+d}{d!}.
\] 
\end{thm}

\par  
In Section~\ref{section_uBounds}, we generalize an argument of De Stefani and the first author that $\fsig(A) \le \frac{1}{2}$ for non-Gorenstein 
Cohen-Macaulay local rings $A$ (see Proposition~\ref{UPforNonGor})
and characterize the case where equality holds. 

\begin{thm}[see Theorem~\ref{EqUPforNonGor}]
Let $A$ be a Cohen-Macaulay local domain with the canonical module $\omega_A$
which is not Gorenstein. 
Then $\fsig(A) \le \frac{1}{2}$, and the following conditions are 
equivalent$:$
\begin{enumerate}
 \item $\fsig(A) = \frac{1}{2}$. 
 \item $A^{1/p^e}$ is a finite direct sum of $A$ 
 and $\omega_A$ for every $e \ge 1$. 
\end{enumerate}
\end{thm}

\par 
When this is the case, $\ehk(A) = \frac{\type(A)+1}{2}$. 
Moreover, if, in addition, either $A$ is $\mathbb{Q}$-Gorenstein or 
a toric singularity, then it is isomorphic to the Veronese subring 
$k[[x_1,x_2,\ldots,x_d]]^{(2)}$, where $k[[x_1,\ldots,x_d]]^{(n)}=k[[(x_1,\ldots, x_d)^n]]$ 
(see Theorem~\ref{Qgor} and \ref{ToricMain}). 

\par \vspace{2mm}
The F-signature of a Gorenstein ring may exceed $\frac{1}{2}$. 
We explore an upper bound on F-signature for  Gorenstein, non-regular local rings of dimension three. 

\begin{thm}[see, Theorem~\ref{Gor-MaxF-sig}]
Let $(A,\m,k)$ be a $3$-dimensional Gorenstein strongly F-regular local ring 
with $\rme(A) \ge 3$. 
Then $\fsig(A) \le \frac{\rme(A)}{24}$. 
\end{thm}

We also provide a classification of (pointed, normal, affine) toric rings with F-signature greater than one half.

\begin{thm}[see Theorem~\ref{maximaltoric}] The toric rings with F-signature greater than $\frac{1}{2}$ are, up to isomorphism, as follows:
\begin{itemize}
\item For a polynomial ring $A$, we have $s(A)=1$.
\item For the coordinate ring $A$ of the Segre product $\bbP^1 \# \bbP^1$, we have $s(A)=2/3$.
\item For the coordinate ring $A$ of the Segre product $\bbP^2 \# \bbP^2$, we have $s(A)=11/20$.
\end{itemize}
\end{thm}

\section{Preliminaries}\label{sec_preliminaries}

Let $(A,\m)$ be a local ring of characteristic $p>0$ and let $F^e \colon A \to A$ denote the $e^{th}$ iterated Frobenius map of $A$. 
For an $A$-module $M$, the Frobenius push-forward of $M$, $F^{e}_{*}M=\{F^{e}_{*}m \mid m \in M\}$, 
is defined as follows: it agrees with $M$ as an abelian group and $A$ acts  
by $a \cdot F^{e}_{*}m =F^{e}_{*}(a^{p^e}m)$ for any $a \in A$ and $m \in M$. 
If $A$ is reduced, $F^{e}_{*}A$ is identified with $A^{1/p^e}$ which consists of $p^e$-th roots of $A$. 
The ring $A$ is called \textit{F-finite} if $F^{e}_{*}A$ is 
a finitely generated $A$-module for every (some) $e \ge 1$. 

We now recall a more general definition of Hilbert--Kunz multiplicity. 
\begin{defn} \label{multiplicity}
Let  $\ell_A(W)$ denote the length of a finitely generated $A$-module $W$.
For an $\m$-primary ideal $I \subset A$ we denote $I^{[q]}=(a^q  \mid a \in I)A$ for each $q=p^e$.
If  $M$ is a finitely generated $A$-module,  
\[
\rme(I, M) \coloneqq \lim_{n \to \infty}\, \dfrac{d!}{n^d} \ell_A(M/I^{n+1}M)
\quad 
\text{(resp.} \;
\ehk(I, M)\coloneqq \lim_{q \to \infty} \dfrac{\ell_A(M/I^{[q]}M)}{q^d} \;\text{)}
\]
is called the \textit{multiplicity} (resp. the 
\textit{Hilbert--Kunz multiplicity}) of $M$ with respect to $I$. 
For brevity, we denote $\rme(I)=\rme(I, A)$ (resp. $\ehk(I)=\ehk(I,A)$) and call it the \textit{multiplicity} 
(resp.\,the \textit{Hilbert--Kunz multiplicity}) of $I$. 
We also denote, $\rme(\m,M)=\rme(M)$ and 
$\ehk(\m, M)=\ehk(M)$.  
\end{defn}

\par \vspace{2mm}
Recall the fundamental properties of Hilbert--Kunz multiplicities; see e.g. \cite{WataYos00}. 

\begin{prop}[\textrm{\cite[(2.3),(2.4),(2.5)]{WataYos00}, \cite{Han03}}] 
\label{Fund-HK}
Let $I \subset A$ be an $\m$-primary ideal.  
\begin{enumerate}
\item The following inequalities hold true$:$
\[
\dfrac{\rme(I)}{d!} \le \ehk(I) \le \rme(I).  
\]
If, in addition, $d \ge 3$, then $\frac{\rme(I)}{d!} < \ehk(I)$. 
\item If $I$ is a parameter ideal, then $\ehk(I)=\rme(I)$. 
\item Let $\Assh(A)$ denote the set of all associated 
prime ideals $P$ with $\dim A/P=\dim A$.  Then
\[
\ehk(I,M) = \sum_{P \in \Assh(A)} \ehk(I, A/P) \cdot \ell_{A_P}(M_P). 
\]
\end{enumerate}
\end{prop}

\subsection{Minimal value of Hilbert--Kunz multiplicity}

Now we want to discuss the conjectural lower bound on Hilbert--Kunz multiplicities of singularities. 
In order to state it, we recall the definition of type $(A_1)$ simple singularity. 

\begin{defn} \label{Apd}
Let $p$ be a prime number, $k$ be an algebraically closed field of characteristic $p$, and $d$ a positive integer. 
Then we define $A_{p,d}$ as follows:
\[
A_{p,d} \coloneqq 
\left\{
\begin{array}{lc}
k[[x_0,x_1,\ldots,x_d]]/(x_0x_1+x_2x_3+\cdots + x_{d-1}x_d) & \text{(when $d=2m-1$, $m \ge 1$)}; \\[2mm]
k[[x_0,x_1,\ldots,x_d]]/(x_0^2+x_1x_2+x_3x_4+\cdots + x_{d-1}x_d) & \text{(when $d=2m$, $m \ge 1$)}.
\end{array}
\right. 
\]
\end{defn}

For $p > 2$ the equation takes a more familiar form $A_{p, d} \cong k[[x_0,x_1,\ldots,x_d]]/(x_0^2+x_1^2+\cdots + x_d^2)$.
Han and Monsky (\cite{HanMon93}) gave an algorithm to compute $\ehk (A_{p,d})$ given $p$ and $d$. However, 
an explicit formula for $\ehk (A_{p,d})$ as a function of $p$ is only known for small values of $d$ (see \cite{Yoshida} for examples).  Gessel and Monsky (\cite{GeMo10}) showed that 
$\lim_{p \to \infty} \ehk(A_{p,d}) = 1+ c_d$ where 
\[
\sec x + \tan x = 1 + \sum_{i=1}^{\infty} c_d\, x^d \quad 
\bigg(|x| < \dfrac{\pi}{2}\bigg).
\]
The first several values of $c_d$ are recorded in Table~\ref{Table of values} below.

\begin{conj}[\textrm{cf. \cite[Conjecture 4.2]{WataYos05}}] \label{LBconj}
Let $(A,\m, k)$ be an F-finite, formally unmixed, non-regular local ring of dimension $d \geq 1$.
Then 
\begin{enumerate}
\item  $\ehk(A) \ge \ehk(A_{p,d}) \ge 1+ c_d$, where $c_d$ is defined above. 
\item  Suppose that $k=\overline{k}$. If $\ehk(A)=\ehk(A_{p,d})$, then $\widehat{A} \cong A_{p,d}$. 
\end{enumerate}
\end{conj}

\par  \vspace{3mm}
Let us summarize the cases where Conjecture~\ref{LBconj} is known. 

\begin{thm} \label{KnownResults}
Let $A$ be a formally unmixed, non-regular local ring and $p$ be an odd 
prime number. 
\begin{enumerate}
\setlength{\itemsep}{5pt}
\item If $d \leq 3$ then Conjecture~\ref{LBconj} holds and $\ehk(A_{p, d}) = 1 + c_d$
\emph{(\cite[Theorem 3.1]{WataYos01a}, \cite[Theorem 3.1]{WataYos05})}. 
In fact, these results also show that $\ehk(A) \ge \ehk(A_{p,d}) = 1+ c_d$  for $p = 2$. 
\item If $d=4$, then Conjecture~\ref{LBconj} holds \emph{(\cite[Theorem 4.3]{WataYos05})} but $\ehk(A_{p,4})=\frac{29p^2+15}{24p^2+12} > \frac{29}{24}$
now depends on $p$ \emph{(\cite{GeMo10})}. 
\item If $d = 5,6$ then $\ehk(A) \ge \ehk(A_{p,d}) \ge 1 + c_d$
\emph{(\cite[Theorem 5.2]{AberEnes13})}.  
\item If $A$ is a complete intersection local ring, then 
$\ehk(A) \ge \ehk(A_{p,d})$ $($see \cite[Theorem 4.6]{EnesShim05}$)$. 
\item The inequality $\ehk(A_{p,d}) \ge 1+ c_d$ in full generality appears in the recent preprint of Trivedi \cite{Trivedi}.
\item Yoshida (\cite{Yoshida}) conjectures that $\ehk(A_{p,d})$ is a decreasing function in $p$ for a fixed $d$.  For $p$ sufficiently large (depending on $d$)
Yoshida's conjecture is also asserted by Trivedi in \cite{Trivedi}.
\end{enumerate}
\end{thm}

\begin{obs}
If $p=2$ and $d=2m$ $(m=1,2,\ldots)$, then  
the following statement can be proved by using an argument in \cite{HanMon93} (see \cite{Yoshida} for details)
\[
\ell \left(A_{2,d}/\m^{[2^e]}\right) = \dfrac{2^m+1}{2^m} \; 2^{de}. 
\]
In particular, $\ehk(A_{2,d})=\frac{2^m+1}{2^m}$. 
\par \vspace{2mm}
Similarly, \cite{Yoshida} conjectures that  
if $d=2m-1$ $(m=1,2,\ldots)$, then 
\[
\ell (A_{2,d}/\m^{[2^e]}) = \frac{2^m}{2^m-1}\; 
2^{de} - \dfrac{(2^{m-1})^e}{2^m-1}
\] 
for every $e \ge 1$. In particular, if would follow that 
$
\ehk(A_{2,d})=\frac{2^m}{2^m-1}. 
$
\end{obs}

Based upon these observations, we pose an improved conjecture as follows:

\begin{conj} \label{Nconj}
Let $(A,\m, k)$ be a formally unmixed non-regular local ring 
of dimension $d \ge 1$ and with algebraically closed residue field.
Let $m \geq 1$ be an integer. 
\begin{enumerate}
\item  If $d=2m-1$,  then either $\widehat{A} \cong A_{p,d}$ or  $\ehk(A) > \frac{2^m}{2^m-1}$. 
\item If $d=2m$, then either $\widehat{A} \cong A_{p,d}$ or  $\ehk(A) > \frac{2^m+1}{2^m}$. 
\end{enumerate}
\end{conj}

By results of Watanabe and Yoshida, Conjecture \ref{Nconj} has an affirmative answer when $p \ge 3$ and $d \le 4$.
The following table depicts the difference between two conjectures. 

\begin{table}[h]
\renewcommand{\arraystretch}{1.5}
\begin{tabular}{|c|c|c|c|c|c|c|} \hline
$d$ & $1$ & $2$ & $3$ & $4$ & $5$ & $6$ \\[1mm] \hline  
$1+c_d$ & $2$ & $\frac{3}{2}$ & $\frac{4}{3}$ & 
$\frac{29}{24}$ & $\frac{17}{15}$ & $\frac{781}{720}$ \\ \hline
$(RHS)$ & $2$ & $\frac{3}{2}$ & $\frac{4}{3}$ & 
$\frac{5}{4}$ & $\frac{8}{7}$ & $\frac{9}{8}$ \\ \hline
\end{tabular}

\caption{Comparison between the two conjectured bounds.}
\label{Table of values}
\end{table}

\subsection{Strong F-regularity and F-signature}
Hilbert--Kunz multiplicity is inherently connected with tight closure, a powerful theory developed by Hochster and Huneke in 
a series of papers starting at \cite{HochHun90}.

\begin{defn}[cf. \cite{HochHun90}] \label{WeakFreg}
Let $I \subset A$ be an ideal, and let $x$ be an element of $A$.  
Put $A^{o}=A  \setminus \cup_{P \in \Min(A)} P$. 
For $x \in A$, we say that $x$ is in the \textit{tight closure} of $I$ (denoted by $I^{*}$) 
if there exists an element $c \in A^{o}$ such that $cx^q \in I^{[q]}$ 
for sufficiently large $q=p^e$. 

A local ring $A$ is said to be weakly F-regular (resp. F-rational) if any ideal $I$ 
(resp. any parameter ideal~$I$) is tightly closed, that is, $I^{*}=I$. 
\end{defn}
\par
A result of Hochster and Huneke \cite[Theorem~8.17]{HochHun90} asserts that $\ehk(I^*) = \ehk(I)$
and, moreover, $I^*$ is the largest ideal containing $I$ with same Hilbert--Kunz multiplicity. 
\par
On the other hand, F-signature 
coincides with the \emph{minimal relative Hilbert--Kunz multiplicity} \cite{WataYos04,Yao05b, PolstraTucker}
and is connected to the following class of singularities. 

\begin{defn}[cf. \cite{HochHun94}] \label{StrFreg}
An F-finite local ring $A$ is called \textit{strongly F-regular} 
if for any $c \in A^{o}$, 
there exists $q = p^e$, $e \ge 1$ such that the map
$A \hookrightarrow A^{1/q}$ defined by $x \mapsto c^{1/q}x$ 
splits as an $A$-linear map. 
Any Noetherian ring $A$ is called  \textit{strongly F-regular}  
if any localization of $A$ is also a strongly F-regular local ring. 
\end{defn}

\par 
Strongly F-regular singularities enjoy many nice properties and are always normal and Cohen-Macaulay. 
For example, quotient singularities and toric singularities are strongly F-regular rings. As it was already mentioned, 
$\fsig(A) > 0$ if and only if $A$ is strongly F-regular by a result of Aberbach and Leuschke \cite[Theorem~0.2]{AberLeu}. The two notions of F-regularity are conjectured to be equivalent and are known to be equivalent in several cases, such as  Gorenstein rings \cite{HochHun89}.

\par 

The simple singularity $A_{p,d}$ discussed above is a hypersurface with $\rme(A_{p,d})=2$, thus by
 \cite[Example~2.3]{WataYos04} $\ehk(A)=2-\fsig(A)$ and $\fsig(A)$ attains the maximal value if  and only if $\ehk(A)$ is minimal.
 The following conjecture is then natural.

\begin{conj} \label{FsigConj}
Let $(A,\m)$ be a non-regular local ring of dimension $d \geq 1$. 
Then  $$\fsig(A)\le 2-\ehk(A_{p,d}) = \fsig(A_{p,d}).$$ 
\end{conj}  

The theory of F-signature originates in the following particular case 
of rings of finite F-representation type, which was introduced by 
Smith and Van den Bergh \cite{SmVand97} (see also \cite{Yao05a}). 

\begin{defn}
We say that $A$ has \textit{finite F-representation type $($FFRT$)$} if 
there is a finite set $\mathcal{S}=\{M_0,M_1,\ldots,M_n \}$ of 
isomorphism classes of indecomposable finitely generated $A$-modules 
such that for any positive integer $e$, $F_{*}^eA$ is isomorphic to 
a finite direct sum of these modules, that is, 
\[
F_{*}^eA \cong M_0^{\oplus c_{0,e}} \oplus M_1^{\oplus c_{1,e}}
\oplus \cdots \oplus 
M_n^{\oplus c_{n,e}}
\] 
for some $c_{i,e} \in \mathbb{Z}_{\ge 0}$. 
Moreover, we say that a finite set $\mathcal{S}=\{M_0,M_1,\ldots,M_n\}$ as above 
is the \textit{$($FFRT$)$ system} of $A$ if every $A$-module $M_i$ appears non-trivially in $F_{*}^eA$ as a direct summand for some $e \in \mathbb{N}$. 
\end{defn}

\section{Lower bound on Hilbert--Kunz multiplicities}
\label{section_lowerHK}

The last two authors gave a lower bound on Hilbert--Kunz 
multiplicities of two-dimensional unmixed (Cohen-Macaulay) local rings
$A$ in terms of usual multiplicities:
\[
\ehk(I) \ge \dfrac{\rme(I)+1}{2} 
\] 
for any $\m$-primary ideal $I$ of $A$ \cite{WataYos01a}. 
In this section, we consider a higher dimensional analogue of this inequality; 
see Theorem \ref{Higher}. 

\par \vspace{2mm}
We recall \cite[Theorem 3.2]{AberEnes13} which improves the volume estimation technique
developed in \cite{WataYos05}.
For any real number $s$ we define $v_{s,d}$ to be 
the volume of $\{(x_1, \ldots, x_d) \in [0,1]^d \mid \sum_{i = 1}^d x_i \leq s\}$ which can be computed as
\[
v_{s,d} = \sum_{n = 0}^{\lfloor s \rfloor} 
(-1)^n \frac{(s - n)^d}{(d-n)!\ n!}, 
\]
where $\lfloor ~\quad ~\rfloor$ stands for round down. 
\par
For an element $x \in A$ we denote 
\[
\overline{\nu}_I(x) \coloneqq \lim_{n \to \infty} \frac{\sup \{k \mid x^n \in I^k\}}{n}.
\]
It is known that the limit exists and 
$\overline{\nu}_I(x) \geq k$ if and only if 
$x \in \overline{I^k}$; 
see Rees \cite{Rees56}.

\begin{thm}[\textrm{Aberbach--Enescu, \cite{AberEnes13}}]
\label{AE}
Let $(A, \m)$ be a formally unmixed reduced local ring of characteristic $p > 0$ and dimension $d$.
Let $J$ be a minimal reduction of an $\m$-primary ideal $I$ and let $r$ be an integer such that 
$r \geq \mu_A(I/J^*)$.
For every real number $s \geq 0$, we have
\[
\ehk(I) \geq \rme(I) \bigg(v_{s,d} - \sum_{i = 1}^r v_{s - t_i,d} \bigg),
\]
where $t_i = \overline{\nu}_I(z_i)$ for $z_1, \ldots, z_r$ generators of $I$ modulo $J^{*}$. 
\par 
In particular, 
\[
\ehk(I) \geq \rme(I) \left(v_{s,d} - r\cdot v_{s - 1,d} \right).
\]
\end{thm}

\par 
Using the above theorem and the technique developed in \cite{AberEnes08}, we can improve Proposition~\ref{Fund-HK}. 

\begin{thm} \label{Higher}
Let $(A,\m)$ be a formally unmixed local ring 
of characteristic $p>0$. 
If $d=\dim A \ge 3$, then for every $\m$-primary ideal $I$ we have 
\[
\ehk(I) > \dfrac{\rme(I)+d}{d!}.
\] 
\end{thm}

\begin{defn} 
A Cohen-Macaulay local ring $(A, \m)$ is said to have \textit{minimal multiplicity}
if $\mu_A(\m)=\rme(A)+\dim A-1$. 
This condition is equivalent to requiring that $\m$ is stable, that is, 
$\m^2=J\m$ for some minimal reduction $J$ of $\m$. 
\end{defn}

In what follows, we may assume that $A$ is complete 
and the residue field $k=A/\m$ is infinite since such extensions do not 
affect multiplicity and Hilbert--Kunz multiplicity.

\begin{lemma}\label{Ineq}
Let $(A,\m,k)$ be a formally unmixed local ring of characteristic $p>0$.
Let $I$ denote an $\m$-primary ideal and $J$ its 
minimal reduction. 
Then $\mu_A(I/J^{*}) \le \rme(I)-1$. Moreover, the equality holds if and only 
if $I = \mf m$ and $\widehat{A}$ is F-rational and has minimal multiplicity. 
\end{lemma}

\begin{proof}
By definition, 
$\mu_A(I/J^{*}) = 
\ell_A(I/J^{*}+\m I)$.
Thus, by colon-capturing (see for example \cite[Lemma~4 and proof of Theorem~6]{MQS}), 
\begin{equation}\label{eq num gen}
\mu_A(I/J^{*}) = 
 \ell_A(A/J^{*}) - \ell_A(A/I) - \ell_A(\m I/J^*\cap\m I)
\le \rme(I) - 1 - \ell_A(\m I/J^*\cap\m I),
\end{equation}
and the first assertion follows.

For the second assertion we recall that the equality $\rme(J) = \ell_A (A/J^*)$ for some parameter ideal characterizes F-rationality by a theorem of Goto and Nakamura \cite{GN}. 
Hence, the equality in (\ref{eq num gen}) forces $A$ to be F-rational and $I = \mf m$. 
However, for $\mu_A(I/J^{*}) = \rme(I) - 1$ we additionally need that $\ell_A(\m I/J^*\cap\m I) = 0$, i.e., 
$\m^2 \subseteq J$. Since $J$ is a minimal reduction, $\m^2 \cap J = \m J$, so 
this condition is equivalent to minimal multiplicity. Note that a complete F-rational ring is necessarily Cohen-Macaulay.
\end{proof}

\par 
The following proposition gives a refinement 
of Aberbach and Enescu \cite[Corollary 3.4]{AberEnes08}. 

\begin{prop} \label{Half}
Let $(A,\m)$ be a Cohen-Macaulay local ring of dimension $d \ge 1$. 
Let $I$ be an $\m$-primary ideal and suppose that 
there exists a minimal reduction $J$ of $I$ 
such that $I^2 \subset J$ $($e.g. $I$ is stable$)$. 
Then 
\[
\ehk(I) \ge \dfrac{\rme(I)}{2}. 
\]
\end{prop}

\begin{proof}
Let $\omega_A$ denote the canonical module of $A$. 
First we observe that $\ehk(J, \omega_A) = \ehk(J)$ for any $\m$-primary ideal $J$.
Namely, we use Proposition~\ref{Fund-HK}(3) after noting 
that $\ell_A ((\omega_A)_P) = \ell_A (\omega_{A_P}) = \ell_A (A_P)$,
by the Matlis duality in the artinian ring $A_P$.

Since $I^{[q]}\omega_A \subseteq J^{[q]} \omega_A \colon I^{[q]}$ for any $q=p^e$ 
by assumption, we get 
\begin{eqnarray*}
\ell_A(\omega_A/J^{[q]}\omega_A) 
&=& \ell_A(\omega_A/J^{[q]}\omega_A \colon I^{[q]})
+\ell_A(J^{[q]}\omega_A \colon I^{[q]}/J^{[q]}\omega_A) \\[1mm]
&\le & \ell_A(\omega_A/I^{[q]}\omega_A) + 
\ell_A(J^{[q]}\omega_A \colon I^{[q]}/J^{[q]}\omega_A).
\end{eqnarray*} 
Then 
\[
\lim_{q\to\infty} \dfrac{\ell_A(\omega_A/J^{[q]}\omega_A)}{q^d}
=\ehk(J, \omega_A)=\ehk(J)=\rme(I), \quad \; 
\lim_{q\to\infty} \dfrac{\ell_A(\omega_A/I^{[q]}\omega_A)}{q^d}
=\ehk(I, \omega_A)=\ehk(I).
\]
On the other hand, since 
\[
(J^{[q]}\omega_A \colon I^{[q]})/J^{[q]}\omega_A 
 \cong  \Hom_{A/J^{[q]}}(A/I^{[q]}, \omega_A/J^{[q]}\omega_A) 
 \cong  \Hom_{A/J^{[q]}}(A/I^{[q]}, \omega_{A/J^{[q]}})
 \cong \omega_{A/I^{[q]}},  
\]
we get 
\[
\ell_A((J^{[q]}\omega_A \colon I^{[q]})/J^{[q]}\omega_A) 
= \ell_A(\omega_{A/I^{[q]}})=\ell_A(A/I^{[q]})
\]
by Matlis duality. 
Hence 
\[
\lim_{q\to\infty} 
\dfrac{\ell_A((J^{[q]}\omega_A \colon I^{[q]})/J^{[q]}\omega_A)}{q^d}
=\ehk(I)
\]
and thus $\rme(I) \le 2 \cdot \ehk(I)$, as required. 
\end{proof}

\begin{cor}[\textrm{\cite{AberEnes08}}] \label{ehk-half}
Let $A$ be a Cohen-Macaulay local ring of dimension $d$ with minimal multiplicity. 
Then $\ehk(A) \ge \frac{\rme(A)}{2}$. 
\end{cor} 
 
\begin{proof}[Proof of Theorem \ref{Higher}]
We may assume that $I$ is tightly closed 
because $\ehk(I)=\ehk(I^{*})$ and $\rme(I)=\rme(I^{*})$. 
Moreover, we may assume $\rme(I) \ge 2$. Let $J$ be a minimal reduction of $I$. 
\begin{flushleft}
{\bf Case 1.} The case of F-rational $A$ and $I^2 \subset J$. 
\end{flushleft}
\par 
We can apply Proposition \ref{Half} to obtain 
$\ehk(I) \ge \dfrac{\rme(I)}{2} > \dfrac{\rme(I)+d}{d!}$ 
if $d \ge 3$ and $\rme(I) \ge 2$. 
\begin{flushleft}
{\bf Case 2.} The remaining case where either $I^2 \not \subset J$ or $A$ is not F-rational. 
\end{flushleft}
\par
By Lemma \ref{Ineq}, we have $\mu_A(I/J^{*}) \le \rme(I)-2$. 
So we can apply Theorem \ref{AE} as $\rme=\rme(I) \ge 2$, 
$r=\rme-2$ and $s=1+\dfrac{1}{\rme}$. 
Then $\lfloor s \rfloor =1$ and 
\begin{eqnarray*}
d! \cdot \rme^d (v_{s, d}-r \cdot v_{s-1, d})
&=& \rme^d \cdot d! \cdot \left(\dfrac{(1+1/\rme)^d}{d!} 
- \dfrac{(1/\rme)^d}{(d-1)!} 
-(\rme-2) \dfrac{(1/\rme)^d}{d!} \right ) \\[1mm]
&=& (\rme+1)^d - d -\rme+2 \\[1mm]
&=& \rme^d+d\rme^{d-1}+\sum_{k=2}^{d-2} {d \choose k} \rme^k 
+ d\rme+1 -d-\rme+2 \\[1mm]
&\ge & \rme^d+d\rme^{d-1}+(d-1)(\rme-1)+2  >  \rme^{d-1}(\rme+d).
\end{eqnarray*}
Hence $\ehk(I) \ge \rme(v_{s,d}-r\cdot v_{s-1, d})> \dfrac{\rme+d}{d!}$, as required. 
\end{proof}

\medskip
If we fix $d$, then this is \textit{not} the best possible. 
In this paper, we prove the following theorem, which gives the optimal bound on the Hilbert-Kunz multiplicity $\ehk(A)$ in dimension $3$.

\begin{thm} \label{3dimOp}
Let $(A,\m,k)$ be a formally unmixed local ring of dimension $3$ and characteristic $p>0$. Then
\[
\ehk(A) \ge \frac{\rme(A)}{6}+1.
\]
If equality holds, then $A$ is a strongly F-regular local ring with 
$\rme(A)=2$. 
Moreover, if, in addition, 
the residue field $k$ is algebraically closed and $p \ge 3$, 
then $\widehat{A} \cong k[[x,y,z,w]]/(xz-yw)$ and $\ehk(A)=\frac{4}{3}$. 
\end{thm}

\par \vspace{2mm}
In order to prove the theorem, we prove a stronger result
as follows. 

\begin{lemma} \label{3dimOPLemma}
Under the assumptions of Theorem $\ref{3dimOp}$, 
we suppose that $\rme =\rme(A) \ge 3$.  If $R$ is F-rational with minimal 
multiplicity then
{\small $\ehk(A) \ge \dfrac{\rme}{6} 
\left(\dfrac{\rme+2+\sqrt{\rme+2}}{\rme+1}\right)^2$.}
If one of these conditions does not hold, 
then we have a stronger bound
\[
\ehk(A) \ge
\dfrac{~1~}{6}
\left(\rme+3+\frac{~2~}{\rme}+\left(2+\frac{2}{\rme}\right)\sqrt{\rme+1} \right). 
\]
\end{lemma}
\begin{proof}
We will optimize the volume estimate from Theorem~\ref{AE}
and we use Lemma~\ref{Ineq} to get a bound on $r = \mu_A(I/J^{*})$.

We now start with the first inequality.
For $1 \le s \le 2$, by Theorem~\ref{AE}
\begin{eqnarray*}
\ehk(A) & \ge & \rme \cdot (v_{s,3}-(\rme-1)v_{s-1,3})\\
&=&\rme \cdot \left (\dfrac{s^3}{6}-\dfrac{(s-1)^3}{2}-(\rme-1)
\dfrac{(s-1)^3}{6}\right ) \\[1mm]
&=&\dfrac{\rme (s^3-(\rme+2)(s-1)^3)}{6}.
\end{eqnarray*}
We consider a function $f(s)=s^3-(\rme+2)(s-1)^3$. 
The derivative is given by 
$f'(s)=3s^2-3(\rme+2)(s-1)^2$ and 
the equation $f'(s)=0$ has roots 
$s_{\pm}=\frac{\rme+2\pm \sqrt{\rme+2}}{\rme+1}$. Since $s_{-} < 1 < s_{+} < 2$, 
the maximum on $1 \le s \le 2$ is at $s_{+}$ 
which gives the inequality:
\[
\ehk(A) \ge \frac{\rme}{6} \cdot f(s_{+})=
\dfrac{\rme}{6}\cdot s_{+}^2=\dfrac{\rme}{6} 
\left(\dfrac{\rme+2+\sqrt{\rme+2}}{\rme+1}\right)^2
\]
\par \vspace{3mm} 

In the second case, we may estimate that
$\ehk(A)  \ge  \rme \cdot (v_{s, 3}-(\rme-2)v_{s-1, 3})$ and $\rme \ge 3$.  
So if we consider $g(s)=s^3-(\rme+1)(s-1)^3$, then 
$1 \le \frac{\rme+1+\sqrt{\rme+1}}{\rme} \le 2$ and 
a similar argument as above implies 
\[
\ehk(A) \ge \dfrac{\rme}{6}\cdot 
g\left(\textstyle{\frac{\rme+1+\sqrt{\rme+1}}{\rme}}\right)
=\dfrac{1}{6}
\left(\rme+3+\frac{2}{\rme}+\left(2+\frac{2}{\rme}\right)\sqrt{\rme+1} \right),
\] 
as required. 
\end{proof}
\vspace{3mm}

\begin{proof}[Proof of Theorem $\ref{3dimOp}$]
First suppose that $A$ is neither F-rational nor 
Cohen-Macaulay with minimal multiplicity. 
If $\rme=2$, then $\ehk(A)=2 > \frac{4}{3}=\frac{\rme}{6}+1$.  
Hence we may assume $\rme =\rme(A) \ge 3$. 
Then Lemma \ref{3dimOPLemma} yields that 
\[
\ehk(A) \ge \dfrac{1}{6}
\left(\rme+3+\frac{~2~}{\rme}+\left(2+\frac{2}{\rme}\right)\sqrt{\rme+1} \right) > \dfrac{\rme}{6}+1. 
\] 
\par 
Next suppose that $A$ is F-rational and Cohen-Macaulay with minimal multiplicity. 
\par \vspace{1mm} \par \noindent
If $\rme \ge 4$, then $\ehk(A) \ge \frac{\rme}{2} > \frac{\rme}{6}+1$.  
\par \vspace{1mm} \par \noindent
If $\rme=3$, 
then \cite[Lemma 3.3(3)]{WataYos05} implies 
$\ehk(A) \ge \frac{13}{8} > \frac{3}{2}=\frac{\rme}{6}+1$. 
\par \vspace{2mm} \par \noindent
Suppose that $\rme=2$. Then the main theorem in \cite{WataYos05} yields 
$\ehk(A) \ge \frac{4}{3}=\frac{\rme}{6}+1$ and 
equality holds if and only if $\widehat{A} \cong k[[x,y,z,w]]/(xz-yw)$. 
Therefore we complete the proof. 
\end{proof}

\section{Upper bounds on F-signature}
\label{section_uBounds}
The main aim of this section is to give an upper bound on F-signature for
non-Gorenstein Cohen-Macaulay local  rings.  We start with a few preliminaries.

We say that $M$ is a maximal Cohen-Macaulay $A$-module
if it is a finitely generated $R$-module such that any (equiv., some)
system of parameters of $R$ is an $M$-regular sequence. 
For such modules $\mu_A(M) \le \rme(M)$, because multiplicity can be computed 
from a regular sequence.  We say that $M$ is an \emph{Ulrich $A$-module} if $\mu_A(M)=\rme(M)$. 
Ulrich modules first appeared in \cite{BHU87} under the name \emph{maximally generated maximal Cohen-Macaulay module}. 

If $A$ is a local ring of positive characteristic $p > 0$ and $M$ is a finitely generated $A$-module
then the rank of the largest free summand of $M$ is independent of a decomposition, because
we may pass to the completion, see \cite[Remark~3.4]{PolstraSmirnov}.
Moreover, if $A$ is a Cohen-Macaulay local ring with the canonical module $\omega_A$ and $M$ is maximal Cohen-Macaulay, then the number of direct summand of $M$ isomorphic to $\omega_A$ is also independent 
of a direct decomposition, since these correspond to a free summand of $\Hom_A (M, \omega_A)$. 
Last, we note that an F-finite Cohen-Macaulay ring has a canonical module by a result of 
Gabber \cite[Remark 13.6]{Gab04}. 

The following proposition is related to Lemma~3.1 of \cite{MPST19}; the second assertion was initially observed by De Stefani and Jeffries in relation with Sannai’s dual F-signature  (\cite{Sannai15}).
Recall that $\type(A)$ is the minimal number of generators of $\omega_A$.

\begin{prop}\label{UPforNonGor}
Let $A$ be an F-finite Cohen-Macaulay local domain which is not Gorenstein. 
Then 
\[
\ehk(A)\le \fsig(A) (\type(A)+1)+2 \cdot \rme(A) 
\left (\frac{1}{2}-\fsig(A) \right ). 
\]
In particular,  $\fsig(A) \le \dfrac{1}{2}$. 
\end{prop}

\begin{proof}
For every $e \ge 1$, we can write 
\[
F^{e}_{*}A=A^{\oplus a_e} \oplus \omega_A^{\oplus b_e} \oplus M_e,
\]
where $a_e$, $b_e$ are non-negative integers and $M_e$ is a maximal 
Cohen-Macaulay $A$-module that does not contain $A$ and $\omega_A$ as direct summands. 
Then 
\[
F_{*}^e\omega_A \cong \Hom_A(F^e_{*}A, \omega_A) \cong 
A^{\oplus b_e} \oplus \omega_A^{\oplus a_e} \oplus \Hom_A(M_e,\omega_A). 
\] 
By the argument in the proof of Sannai \cite[Proposition 3.10]{Sannai15}, 
we have 
\[
\lim_{e \to \infty} \frac{a_e}{\rank F_*^e A} = \lim_{e \to \infty} \frac{b_e}{\rank F_*^e A} = \fsig(A).
\]

\par 
Since $M_e$ is a maximal Cohen-Macaulay $A$-module, we then have 
\begin{eqnarray}
\label{key_inequation}
\mu_A(F^{e}_{*}A) & =& a_e + b_e \cdot \type(A) + \mu_A(M_e) \\ \nonumber
& \le &  a_e + b_e \cdot \type(A) +\rme_A(M_e) \\ \nonumber
&=&  a_e + b_e \cdot \type(A) +\rme(A) \rank_A (M_e) \\ \nonumber
&=& a_e + b_e \cdot \type(A) +\rme(A) (\rank F_*^e A-a_e-b_e). 
\end{eqnarray}
Hence 
\[
\dfrac{\ell_A(A/\frm^{[p^e]})}{p^{ed}} = 
\dfrac{\mu_A(F^{e}_{*}A)}{\rank F^{e}_{*}A}
\le 
\dfrac{a_e}{\rank F^{e}_{*}A} + \type(A) 
\cdot \dfrac{b_e}{\rank F^{e}_{*}A}
+ \rme(A) \left(1 - \dfrac{a_e}{\rank F^{e}_{*}A} 
- \dfrac{b_e}{\rank F^{e}_{*}A} \right), 
\]
and the first assertion follows after taking limits as $e$ tends to $\infty$.

In particular, since 
\[
0 \le \dfrac{\rank_A(M_e)}{\rank F^{e}_{*}A} 
= 1 - \dfrac{a_e}{\rank F^{e}_{*}A} - \dfrac{b_e}{\rank F^{e}_{*}A}, 
\]
we get $0 \le 1 - 2 \cdot \fsig(A)$, that is, $\fsig(A) \le \frac{1}{2}$. 
\end{proof}

\begin{remark}

We note that there are examples of Gorenstein rings having F-signature at least $\frac{1}{2}$.
The first example is a $2$-dimensional Gorenstein strongly F-regular local ring: 
$A$ is necessarily a hypersurface and has minimal multiplicity, thus $\rme(A)=2$. 
Therefore, we have $\fsig(A)=2-\ehk(A)\ge 2-\frac{3}{2}=\frac{1}{2}$ by  \cite[Corollary~2.6]{WataYos05} 
(see also \cite[Example~4.1]{WataYos01a} and \cite[Example~18]{HunLeu}).
Second, the Segre product $k[x,y] \# k[a,b]$ has F-signature $\frac 23$ by  Proposition~\ref{ToricFFRT3}. 
\end{remark}

The following theorem characterizes the equality in Proposition~\ref{UPforNonGor}. 

\begin{thm} \label{EqUPforNonGorUlrich}
Let $A$ be a strongly F-regular local domain which is not Gorenstein. 
The following conditions are equivalent$:$
\begin{enumerate}
\item $\ehk(A) = \fsig(A) (\type(A)+1)+2 \cdot \rme(A) \left (\frac{1}{2}-\fsig(A) \right )$. 
\item Each $F^{e}_{*}A$ is a finite direct sum of $A$,  $\omega_A$ and an Ulrich $A$-module $M_e$. 
\end{enumerate}
\end{thm}

\begin{proof}
For every $e \ge 1$, we can write 
$F^{e}_{*}A=A^{\oplus a_e} \oplus \omega_A^{\oplus b_e} \oplus M_e$, 
where $a_e$ and $b_e$ are non-negative integers and $M_e$ is a maximal 
Cohen-Macaulay $A$-module that does not contain $A$ and $\omega_A$ as direct summands. 
\par \vspace{2mm } \par \noindent
$(2) \Longrightarrow (1)$: 
By the assumption, $M_e$ is an Ulrich $A$-module, that is, $\mu_A(M_e)=\rme(M_e)$. 
Hence the assertion follows from the proof of Proposition~\ref{UPforNonGor}. 
\par \vspace{2mm } \par \noindent
$(1) \Longrightarrow (2)$: 
Suppose that there exists $e'$ such that $F_{*}^{e'}A=A^{\oplus a_{e'}} \oplus \omega_A^{\oplus b_{e'}} \oplus M_{e'}$, 
where $M_{e'}$ is an maximal Cohen-Macaulay $A$-module but not an Ulrich $A$-module,  namely, $\mu_A(M_{e'})< \rme(M_{e'})$. 
By \cite[Lemma~3.3]{PolstraSmirnov}  we may now build a similar decomposition for all $e \geq e'$:
\[
F_*^e A = A^{\oplus a_{e}} \oplus \omega_A^{\oplus b_{e}} \oplus  M_{e'}^{\oplus c_e} \oplus N_e
\] 
such that $\liminf_{e \to \infty} c_e/\rank F_*^e A > 0$.
Following the proof of  Proposition~\ref{UPforNonGor} we obtain that 
\[
\mu_A(F^{e}_{*}A) \leq 
a_e + b_e \cdot \type(A) - c_e(\rme(M_e) - \mu_A(M_e)) + \rme(A) (\rank F_*^e A-a_e-b_e),  
\]
which shows after dividing by $\rank F_*^e A$ and passing to the limit that
\[
\ehk(A) < \fsig(A) (\type(A)+1)+\rme(A) (1-2 \cdot \fsig(A)). \qedhere
\]
\end{proof}

\par
One can prove the following proposition by a similar method as in the proof 
of Proposition \ref{UPforNonGor} and Theorem \ref{EqUPforNonGorUlrich}. 

\begin{prop}\label{UPforGor}
Suppose that $A$ is an F-finite Gorenstein local ring of dimension $d \ge 2$.  
Then 
\[
\ehk(A) \le \fsig(A) + (1-\fsig(A))\cdot \rme(A)
\]
and the equality holds if and only if for all $e\geq 1$ $F_{*}^eA$ can be written as 
a direct sum of $A$ and Ulrich $A$-modules. 
\end{prop}

We note that if $\rme(A)=2$ we have $\ehk(A)=2-\fsig(A)$, and hence $A$ satisfies the equality of Proposition~\ref{UPforGor}. 

\begin{quest}
If $A$ is Gorenstein and satisfies $\ehk(A) = \fsig(A) + (1-\fsig(A))\cdot \rme(A)$, then is $\rme(A) \le 2$?
\end{quest}

We proceed to study non-Gorenstein rings whose F-signature is $\frac{1}{2}$.

\begin{thm} \label{EqUPforNonGor}
Let $A$ be an F-finite Cohen-Macaulay local domain which is not Gorenstein. 
Then the following conditions are equivalent$:$
\begin{enumerate}
 \item $\fsig(A) = \frac{1}{2}$, 
 \item $A$ is FFRT with the FFRT system $\{A, \omega_A\}$.  
\end{enumerate}
When this is the case, $\ehk(A) = \frac{\type(A)+1}{2}$. 
\end{thm}

\begin{proof}
$(2) \Longrightarrow (1)$ essentially follows from the proof of Proposition \ref{UPforNonGor}, because in this case there is no $M_e$ and we have equality throughout.  
\par 
$(1) \Longrightarrow (2)$:  Assume that for some $e' \ge 1$, we write $F^e_{*}A$ as 
\[
F^{e'}_{*}A=A^{\oplus a_{e'}} \oplus \omega_A^{\oplus b_{e'}} \oplus M_{e'}, 
\]
where $0 \neq M_{e'}$ is a maximal Cohen-Macaulay $A$-module that does not have $A$ and $\omega_A$ as direct summands. 
Since $R$ is strongly F-regular by the assumption, as explained in \cite[Lemma~3.3]{PolstraSmirnov}  we may now build similar decompositions for $e \geq e'$:
\[
F_*^e A = A^{\oplus a_{e}} \oplus \omega_A^{\oplus b_{e}} \oplus  M_{e'}^{\oplus c_e} \oplus N_e
\] 
with $\liminf_{e \to \infty} c_e/\rank F_*^e A > 0$. After taking ranks we then have that
\[
1 \ge \dfrac{a_e}{\rank F_*^e A} + \dfrac{b_e}{\rank F_*^e A} + \rank_A M_{e'}\dfrac{c_e}{\rank F_*^e A} 
\]
which  after taking limits then gives that $1 > \fsig (A)+ \fsig(A)$ which is a contradiction. 
\end{proof}

Let us give an example of local rings having $\fsig(A)=\frac{1}{2}$. 

\begin{exam}  \label{second-ex}
Let $d \ge 2$ be an integer. 
Let $A=k[[x_1,\ldots, x_d]]^{(2)}$ be the second Veronese subring of 
the formal power series ring over $k$. 
Then $\fsig(A)=\frac{1}{2}$. 
Moreover, $A$ is not Gorenstein if and only if $d$ is odd. 
\end{exam}

\par 
Let $A$ be a Cohen-Macaulay local domain with minimal multiplicity. 
Then $A$ is not Gorenstein if and only if $\rme(A) \ge 3$. Moreover, then $\type(A)=\rme(A)-1$. 
So we can obtain the following corollary 
by combining \ref{ehk-half}, \ref{UPforNonGor}, \ref{EqUPforNonGor} and \ref{EqUPforNonGorUlrich}.

\begin{cor}
Suppose that $A$ is a Cohen-Macaulay local domain with minimal multiplicity and with $\rme(A) \ge 3$. 
Then 
\begin{enumerate}
\item $\fsig(A) \le \frac{1}{2}$. 
\item $\frac{\rme(A)}{2} \le \ehk(A) \le (1-\fsig(A))\rme(A)$.  
\item The following conditions are equivalent$:$
\begin{enumerate}
\item $\fsig(A) = \frac{1}{2}$. 
\item $A$ has FFRT with the FFRT system $\{A, \omega_A\}$.
\end{enumerate}
When this is the case, $\ehk(A) = \frac{\rme(A)}{2}$. 
\item Suppose $\fsig(A)>0$. Then the following conditions are equivalent$:$
\begin{enumerate}
\item $\ehk(A)=(1-\fsig(A))\rme(A)$. 
\item $F_{*}^eA$ can be written as a direct sum of $A$, $\omega_A$ and Ulrich $A$-modules for every $e \ge 1$. 
\end{enumerate}
\end{enumerate}
\end{cor}

\begin{exam}
Let $A=k[[x,y,z]]^{(2)}$. 
Then $A$ has minimal multiplicity and its multiplicity is $4$. 
Moreover, $\ehk(A)=\frac{\rme(A)}{2}=2$ and $\fsig(A)=\frac{1}{2}$.   
\end{exam}

\begin{exam}
Let $A=k[[x^3,xy^2,xy^2,y^3]]=k[[x,y]]^{(3)}$. Then, $F_{*}^eA$ can be written as 
direct sum of $A$, $\omega_A=Ax+Ay$ and $M=Ax^2+Axy+Ay^2$. 
In this case, $\fsig(A)=\frac{1}{3}$ by \cite{WataYos04}, and $\rme(A)=3$, $\type(A)=2$. 
Since $\mu_A(M)=3=\rme_A(M)=\rme(A)\rank_A(M)$, we see that $M$ is an Ulrich $A$-module, 
and we have 
\[
\ehk(A)=2=\fsig(A) (\type(A)+1)+2 \cdot \rme(A) \left (\frac{1}{2}-\fsig(A)\right). 
\]
\end{exam}

We pose the following question. 

\begin{quest} \label{quest-half}
Let $A$ be a $d$-dimensional Cohen-Macaulay local domain with isolated singularity and that $\fsig(A)=\frac{1}{2}$. 
Then is $A$ isomorphic to the ring defined in Example~\ref{second-ex}?
\end{quest}

\subsection{$\mathbb{Q}$-Gorenstein local rings}
We are able to give an affirmative answer to Question~\ref{quest-half} in a particular case.
 
 Let $A$ be a Cohen-Macaulay reduced local ring. 
For an ideal $I \subset A$ of pure height $1$, the $n^{th}$ symbolic power $I^{(n)}$ 
denotes the intersection of 
height one primary components of $I^n$.   

\begin{defn} \label{QGorDefn}
Let $A$ be a normal local domain having a canonical module $\omega_A$. 
\begin{enumerate}
\item 
The ring $A$ is said to be \textit{$\mathbb{Q}$-Gorenstein} if there exists an ideal 
$J$ of pure height $1$ which is isomorphic to $\omega_A$ as an $A$-module such 
that $J^{(n)}$ is principal. Furthermore, 
\[
\Index(A) \coloneqq\mathrm{min}\{n \ge 1 \mid \text{$J^{(n)}$ is principal}\}
\] 
is called the \textit{index} of $A$. Note that $A$ is Gorenstein if and only if it is $\mathbb{Q}$-Gorenstein of index $1$. 
\item Suppose that $A$ is a $\mathbb{Q}$-Gorenstein normal local domain of 
$r=\Index(A) \ge 2$, and let $J$ be an ideal such that $J \cong \omega_A$. 
Then the \textit{canonical cover} of $A$ is defined as 
\[
B\coloneqq \bigoplus_{i=0}^{r-1} J^{(i)}. 
\]
\end{enumerate}
\end{defn}

\par

We will use the following result due to Carvajal-Rojas \cite[Theorem~C]{CR}
which extends prior work \cite[Theorem 2.6.5]{VonKorff} and \cite{Wata91}.

\begin{thm} \label{cover}
Let $A$ be a strongly F-regular $\mathbb{Q}$-Gorenstein 
local domain and let $r \ge 2$ be the index of $A$.
Let $B$ be the canonical cover of $A$. 
Then we have $\fsig(A)=\frac{1}{r} \fsig(B)$. 
\end{thm}

Using this, we can prove the following theorem. 

\begin{thm} \label{Qgor}
Let $A$ be an F-finite $\mathbb{Q}$-Gorenstein normal local domain of 
characteristic $p>0$. Assume the index $r$ of $A$ is at least $2$. 
Then the following conditions are equivalent$:$ 
\begin{enumerate}
\item $\fsig(A)=\frac{1}{2}$. 
\item $r=2$ and $A$ admits a canonical cover $B$ which is regular. 
\end{enumerate}
\end{thm}

\begin{proof}
Suppose (1). Let $B$ a canonical cover of $A$. Then $\fsig(A)=\frac{\fsig(B)}{r}$. 
If $r \ge 3$, then $\fsig(B)=r\cdot \fsig(A)> 1$. This is a contradiction. 
Hence $r =2$ and $\fsig(B)=1$. Hence $B$ is regular. 
The converse is easy. 
\end{proof}

\subsection{The F-signatures of $3$-dimensional Gorenstein rings}
\label{subsection_3dim_maxFsig}
We want to present a few upper bounds on F-signature in view of Question~\ref{que_maxFsig}.
We first estimate F-signature using the multiplicity. 

\begin{thm} \label{Gor-MaxF-sig}
Let $(A,\m,k)$ be a $3$-dimensional Gorenstein strongly F-regular local ring 
with multiplicity ${\rme(A) \ge 3}$. 
Then $\fsig(A) \le \frac{\rme(A)}{24}$. 
\end{thm}

\begin{proof}
Let $J$ be a minimal reduction of $\m$. 
Then we can write $J \colon \m=(J,u)$ for some 
$u \in \m \setminus J$ because $A$ is Gorenstein. 
Moreover, we have 
\[
\fsig(A) \le \ehk(J)-\ehk(J\colon \m)=\rme(J)-\ehk(J \colon \m)
=\rme(A)-\ehk(J\colon \m). 
\] 
Since $A$ is strongly F-regular, applying the Brian\c con--Skoda theorem to 
$J^{[q]}$ we get that 
$\m^{3q} \subset J^{[q]}$, and thus $\m^2 \subset J \colon \m$. 
Since $A$ is not double point, $\m^2 \not \subset J$. 
Hence there exists an element $v \in \m^2$ such that 
$v \in J \colon \m \setminus J$. 
Write $v=a+ru$ for some $a \in J$ and $r \in A$. 
Suppose $r \in \m$. 
Then $ru \in J$ and thus $v=a+ru \in J$, which is a contradiction. 
Hence $r \in A \setminus \m$ and $(J,u)=(J,v)$. 
So we may assume that $u \in \m^2$. 
Then $u^q \in \m^{2q}$ and we get 
\begin{eqnarray*}
\ell_A \left(\dfrac{u^qA+J^{[q]}}{J^{[q]}}\right) 
 \le  \ell_A 
\left(\dfrac{u^qA+\m^{\lfloor\frac{5}{2}q\rfloor }+J^{[q]}}{J^{[q]}}\right) 
=
 \ell_A \left(\dfrac{u^qA+\m^{\lfloor\frac{5}{2}q\rfloor }+J^{[q]}}{\m^{\lfloor\frac{5}{2}q\rfloor }+J^{[q]}}\right) + \ell_A \left(\dfrac{\m^{\lfloor\frac{5}{2}q\rfloor }+J^{[q]}}{J^{[q]}}\right).
\end{eqnarray*}
Note that $\overline{A}=A/J^{[q]}$ is an Artinian Gorenstein local ring.
Thus the Matlis duality yields that 
\[
\ell_A \left(\dfrac{u^qA+\m^{\lfloor\frac{5}{2}q\rfloor }+J^{[q]}}{\m^{\lfloor\frac{5}{2}q\rfloor }+J^{[q]}}\right)  \le \ell_A(A/\m^{\lceil\frac{1}{2}q\rceil })
\quad \text{and} \quad 
\ell_A \left(\dfrac{\m^{\lfloor\frac{5}{2}q\rfloor }+J^{[q]}}{J^{[q]}}\right) 
\le \ell_A(A/\m^{\lceil\frac{1}{2}q\rceil }).
\]
Therefore 
\[
\fsig(A) \le \lim_{e \to \infty} \ell_A \left(\dfrac{u^qA+J^{[q]}}{J^{[q]}}\right)/q^3
\le 2 \cdot  \lim_{e \to \infty} 
\dfrac{\ell_A(A/\m^{\lceil \frac{1}{2}q\rceil })}{q^3} = 
2 \times \dfrac{1}{3!} \left(\dfrac{1}{2}\right)^3
\rme(A) =\dfrac{1}{24}\rme(A),
\]
as required. 
\end{proof}

\par 
The next example shows that Theorem~\ref{Gor-MaxF-sig} gives the best possible 
bound. 

\begin{exam} \label{EX-second}
Let $R^{(2)}$ be the $2^{nd}$ Veronese subring of $R=k[x,y,z,w]/(xw-yz)$. 
Set 
\[
A=k[[(x,y,z,w)^2]]/(xw-yz),
\] 
which is the completion with respect to 
the irrelevant maximal ideal $\m$ of $R^{(2)}$. 
Then $A$ is a $3$-dimensional Gorenstein strongly F-regular local domain. 
Hence \cite[Theorem 5.1(1)]{HunWat15} implies $\rme(A) \le \emb(A)-1=8$. 
On the other hand, since $A$ is not hypersurface (of multiplicity $2$), 
$\rme(A) \ge \emb(A)-\dim A+2 = 9-3+2=8$ and thus $\rme(A)=8$. 
Moreover, we have 
\[
\fsig(A)= \dfrac{\fsig(R)}{2} = \dfrac{2/3}{2}=\dfrac{1}{3}=\dfrac{\rme(A)}{24}. 
\]
On the other hand, by \cite[Corollary 1.10]{WataYos01a}, we have 
\[
\ehk(\m_R^2)=\rme(\m){2 +3-2 \choose 3} + \ehk(\m){2+3-2 \choose 3-1}
=\rme(R)+3 \cdot \ehk(R)=2+3 \cdot \frac{4}{3}=6. 
\] 
Hence 
\[
\ehk(A)=\dfrac{\ehk(\m_A R)}{2}=\dfrac{\ehk(\m_R^2)}{2}=\dfrac{6}{2}=3 
> \frac{7}{3}=\frac{8}{6}+1=\frac{\rme(A)}{6}+1.
\]
\end{exam}

\begin{remark}
The argument given in the proof of Theorem~\ref{Gor-MaxF-sig} is also valid for 
some classes of higher dimensional Gorenstein rings. 
For example, let $A\coloneqq k[[x_0,x_1,\ldots,x_d]]/(x_0^d+x_1^d+\cdots+x_d^d)$. 
For the maximal ideal $\m$ and its minimal reduction J, we have that $\m^d\subset J$ and $\m^{d-1}\not\subset J$. 
Thus, by the same argument as the proof of Theorem~\ref{Gor-MaxF-sig}, 
we see that $\fsig(A)\le \frac{\rme(A)}{2^{d-1}d!}=\frac{1}{2^{d-1}(d-1)!}$, 
see also \cite[Proposition~2.4 and Question~2.6]{WataYos04}. 
\end{remark}

As the first open case, we will investigate  Question~\ref{que_maxFsig}, Conjecture~\ref{FsigConj} for $3$-dimensional Gorenstein rings.
In particular, we ask the following question. 

\begin{quest} 
\label{Quest-MaxF-sig}
Let $(A,\m)$ be a $3$-dimensional non-regular 
Cohen-Macaulay local ring. 
Is $\fsig(A) \le \frac{2}{3}$?
\end{quest}  

If this is correct, then this bound is best possible because if $A = k[x,y,z,w]/(xw - yz )$, then ${\fsig(A)= \frac{2}{3}}$. 
We will give a positive answer to this question for the case of toric rings in the next section (see Theorem~\ref{maximaltoric}). 
For a general situation, we only have the inequality given in Proposition~\ref{UpperBound}. 

\begin{prop} \label{UpperBound}
Let $A$ be a $3$-dimensional strongly F-regular local domain which is not regular. 
Then $\fsig(A) < \frac{5}{6}$. 
\end{prop}

\begin{proof}
We may assume that $A$ is Gorenstein (see Proposition~\ref{UPforNonGor}). 
By Proposition~\ref{UPforGor}(1), we have 
\[
\ehk(A) \le \fsig(A)+\rme(A)(1-\fsig(A)). 
\]
On the other hand, Theorem \ref{3dimOp} implies that
\[
\frac{\rme(A)}{6}+1 \le \ehk(A) \le  \fsig(A)+\rme(A)(1-\fsig(A)). 
\]
Hence 
\[
\rme(A) \left(\fsig(A) - \frac{5}{6}\right) \le \fsig(A)-1 < 0.
\]
and thus $\fsig(A) < \frac{5}{6}$. 
\end{proof}

\section{Observations on toric rings}
\label{sec_toric}

In this section, we further study an upper bound on F-signature of a toric ring. 
In particular, in Theorem~\ref{maximaltoric} we give a positive answer to Question~\ref{Quest-MaxF-sig}.

\subsection{Preliminaries}
\label{toric_preliminaries}
Let $\sfN \cong \bbZ^d$ be a lattice of rank $d$. 
Let $\sfM=\Hom_{\bbZ}(\sfN,\bbZ)$ be the dual lattice of $\sfN$. 
We set $\sfN_{\bbR}=\sfN \otimes_\bbZ \bbR$ and $\sfM_{\bbR} = \sfM \otimes_\bbZ \bbR$. 
We denote the inner product by $\langle \quad, \quad \rangle \colon 
\sfM_\bbR \times \sfN_\bbR \to \bbR$. 
Let 
\[
\sigma \coloneqq \bbR_{\ge 0} v_1 + \cdots + \bbR_{\ge 0} v_n \subset \sfN_{\bbR}
\]
be a strongly convex rational polyhedral cone of dimension $d$ generated by 
$v_1,\ldots,v_n \in \bbZ^d$ where $d \le n$. 
We assume that $v_1,\ldots,v_n$ are minimal generators of $\sigma$. Namely, the cone generated by any proper subset of $\{v_1,\ldots,v_n\}$ is a proper subset of $\sigma$, and each $v_i$ is the lattice point of smallest magnitude on the ray it generates.
For each generator, we define the linear form $\lambda_i (-) \coloneqq \langle - ,v_i\rangle$. 
We consider the dual cone $\sigma^{\vee}:$ 
\[
\sigma^{\vee} \coloneqq \{{\bf x} \in \sfM_{\bbR} \mid \lambda_i({\bf x}) \ge 0 \;\text{ for all }\, i=1,2,\ldots,n\}. 
\]
In this case, $\sigma^{\vee} \cap \sfM$ is a positive normal affine monoid. 
Given  an algebraically closed field $k$ of characteristic $p>0$, 
we define the toric ring 
\[
A\coloneqq k[\sigma^{\vee} \cap \sfM] = k[t_1^{m_1}\cdots t_d^{m_d} \mid (m_1,\ldots,m_d) \in \sigma^{\vee} \cap \sfM].
\]
Thus, in this paper, a toric ring is a pointed normal affine monomial ring.
In particular, any toric ring that we consider is strongly F-regular. We denote the irrelevant ideal of $A$ as $\frm$.

\par 
For each ${\bf a}=(a_1,\ldots,a_n) \in \bbZ^n$, we set 
\[
V({\bf a})\coloneqq\{{\bf x} \in \sfM \mid (\lambda_1({\bf x}),\ldots, \lambda_n({\bf x})) \ge 
(a_1,\ldots,a_n)\}.
\]
Then we define the divisorial ideal (rank one reflexive module) $D({\bf a})$  
generated by all monomials whose exponent vectors are in $V({\bf a})$. 
For example, we have that $R=D({\bf 0})$ and $\omega_A \cong D(1,1,\ldots,1)$. 
Let $\frp_i:=D(\delta_{i1},\ldots,\delta_{in})$, where $\delta_{ij}$ is the Kronecker delta. 
The height one prime ideals $\frp_1,\ldots,\frp_n$ generate 
the class group $\Cl(A)$. 
When we consider a divisorial ideal $D({\bf a})$ as the element of $\Cl(A)$, we denote it by $[D({\bf a})]$. 

\par \vspace{2mm}
In what follows, we will pay attention to a certain class of divisorial ideals called conic. 

\begin{defn}[\textrm{see e.g. \cite{BruGub03, Bru05}}]
We say that a divisorial ideal $D({\bf a})$ is \textit{conic} if there exist 
${\bf x} \in \sfM_{\bbR}$ such that ${\bf a}=(\ulcorner \lambda_1({\bf x}) \urcorner, \ldots, 
\ulcorner \lambda_n({\bf x}) \urcorner)$, where $\ulcorner \quad \urcorner$ stands for the round up. 
\end{defn}

Any conic divisorial ideal is a rank one maximal Cohen-Macaulay module (see \cite[Corollary~3.3]{BruGub03}). 
We denote the set of isomorphism classes of conic divisorial ideals of a toric ring $A$ by 
$\mathcal{C}(A)$. This is a finite set because the number of isomorphism 
classes of rank one maximal Cohen-Macaulay $A$-modules is finite (see \cite[Corollary 5.2]{BruGub03}). 
The following proposition guarantees that any conic divisorial ideal appears in $F_{*}^eA$ 
as a direct summand.

\begin{thm}[\textrm{\cite[Proposition 3.6]{BruGub03}, \cite[Subsection 3.2]{SmVand97}}]
Let $A$ be a toric ring as above. Then, $A$ has FFRT by the FFRT system $\mathcal{C}(A)$. 
\end{thm}

We recall that our arguments can be reduced to the $\frm$-adic completion of $A$, 
as we mentioned in the beginning of Section~\ref{section_uBounds}. 
Thus, we may assume that $A$ is complete local, in which case the Krull--Schmidt condition holds for $A$.

\begin{remark}
\label{toric_remark}
In some parts of this section, we assume that if the class group $\Cl(A)$ contains 
a torsion element, then the order of that element is coprime to $p$. 
In this case, the toric ring $A$ is a ring of invariants. 

Namely, let $k^{\times}$ be the multiplicative group of $k$ and
$G\coloneqq\Hom(\Cl(A),k^{\times})$ be  the character group of $\Cl(A)$.
The group $G$ acts on $B\coloneqq k[x_1,\ldots,x_n]$ by $g \cdot x_i=g([\frp_i])x_i$ for each $g \in G$ and any $i$. 
Then, by \cite[Theorem 2.1(b)]{BruGub03}, $A$ can be described as $A \cong B^{G}$. 
Moreover, to avoid the triviality, we assume that $g([\frp_i])\neq1$ for any $i$, that is, $[\frp_i]\neq0$ in $\Cl(A)$. 
\end{remark}

Last, we will use that the F-signature of a toric ring can be computed combinatorially and, in particular, does 
not depend on the characteristic.

\begin{thm}[{\cite[Theorem~5.1]{WataYos04}, see also \cite{Bru05,Sin05,VonKorff}}]
\label{Fsig_toric}
Let $A$ be a toric ring. Then, we may compute 
\[
\fsig(A)=\mathrm{vol}\{\mathbf{x}\in\sfM_\bbR \mid 0\le \lambda_i(\mathbf{x})\le 1 \,\, \text{for all $i$}\}. 
\]
\end{thm}

\subsection{Cohen-Macaulay toric rings}

We recall that the F-signature of non-Gorenstein ring is less than or equal to $\frac{1}{2}$ (see Propsition~\ref{UPforNonGor}). 
We now determine the non-Gorenstein toric rings whose F-signature is~$\frac{1}{2}$. 

\begin{prop} \label{ToricFFRT}
Let $A$ be a toric ring as in Remark~\ref{toric_remark}. 
Then, the following conditions are equivalent. 
\begin{enumerate}
\item The FFRT system of $A$ is $\{A, M\}$ with $M \not \cong A$. 
\item $A$ is isomorphic to the Veronese subring $k[x_1,\ldots,x_d]^{(2)}$ of degree $2$. 
\end{enumerate}
When this is the case, the F-signature is $\fsig(A)=\frac{1}{2}$. 
\end{prop}

\begin{proof}
We first assume that the FFRT system of $A$ is $\{A, M\}$. 
We note that $M \cong \omega_A$ if $A$ is not Gorenstein. In fact, $\omega_A$ certainly appears 
in $F_{*}^eA$ as a direct summand for sufficiently large $e$, 
since $A$ is strongly F-regular (cf. \cite[Proof of Proposition 3.10]{Sannai15}, \cite[Proposition~2.1]{HigNak2}). 
By \cite[Remark 3.4]{BruGub03}, the divisorial ideals $\frp_1,\ldots,\frp_n$ are conic. 
Since $\mathcal{C}(A)=\{A,M\}$, we have that $\frp_1\cong \cdots \cong \frp_n \cong M$. 
Thus we have that $[\frp_1]=\cdots = [\frp_n]=[M]$ in $\Cl(A)$. 
On the other hand, since $[\omega_A]=[\frp_1]+\cdots + [\frp_n]$, we see that $n[\frp_i]=0$ 
(resp. $(n-1)[\frp_i]=0$) in $\Cl(A)$ for any $i$ if $A$ is Gorenstein (resp. 
not Gorenstein). 
Thus we conclude that $\Cl(A)$ is a finite cyclic group generated by $[\frp_i]$, 
that is, $\Cl(A) \cong \langle [\frp_i] \rangle$. 
By \cite[Corollary~4.59]{BruGub09}, this implies that the cone $\sigma$ defining $A$ is simplicial (i.e., $n=d$), 
thus we have that $A \cong B^{G}$ where $B=k[x_1,\ldots,x_d]$ and $G=\Hom(\Cl(A),k^{\times})$ is a finite cyclic group.
We may assume that $G$ is small (see e.g. \cite[Proof of Theorem 5.7]{IW14}). 
By \cite[Proposition 3.2]{SmVand97}, each indecomposable direct summand of $F_{*}^eA$ is 
a module of covariants which takes the form $(B \otimes_k V_i)^G$, where $V_i$ is an irreducible 
representation of $G$. Since the FFRT system is $\{A,M\}$ and $G$ is small, we have only  
two non-isomorphic irreducible representations of $G$. 
Then we have that $|G|=2$, and the elements of $G$ are the characters of $\rho_0,\rho_1 
\in \Hom(\Cl(A),k^{\times})$ defined by $\rho([\frp_i])=1$ and 
$\rho([\frp_i])=-1$ respectively. 
Consequently, we have that $A \cong k[x_1,\ldots,x_d]^{(2)}$. By \cite[Theorem 4.2]{WataYos04}, the F-signature of $A\cong B^G$ is $\frac{1}{|G|}=\frac{1}{2}$. 
\par 
On the other hand, we assume that $A\cong k[x_1,\ldots,x_d]^{(2)}$. Then, $A$ is the invariant subring of $k[x_1,\ldots,x_d]$ 
under the action of the cyclic group $\langle g=\text{\rm diag}(-1,\ldots,-1)\rangle$ defined by $g\cdot x_i=-x_i$ for any $i$. 
Thus, the condition $(1)$ follows from \cite[Proposition~3.2.1]{SmVand97}. 
\end{proof}

This is the main result in this subsection. 

\begin{thm} \label{ToricMain}
Let $A$ be a toric ring as Remark~\ref{toric_remark}. 
Assume that $A$ is not Gorenstein, then the following conditions are equivalent. 
\begin{enumerate}
\item $\fsig(A)=\frac{1}{2}$. 
\item The FFRT system of $A$ is $\{A,\omega_A\}$. 
\item $A$ is isomorphic to the Veronese subring $k[x_1,\ldots,x_d]^{(2)}$ of degree $2$, 
where $d$ is an odd number. 
\end{enumerate}
\end{thm}

\begin{proof}
$(3) \Longrightarrow (1)$ follows from \cite[Theorem 4.2]{WataYos04}. 
$(1) \Longleftrightarrow (2)$ follows from Theorem~\ref{EqUPforNonGor}.
Then we show $(2) \Longleftrightarrow (3)$. 
By Proposition \ref{ToricFFRT}, we see that $A \cong k[x_1,\ldots,x_d]^{(2)} = k[x_1,\ldots,x_d]^{G}$, where $G \cong \langle \mathrm{diag}(-1,\ldots,-1)\rangle$. 
Since $A$ is not Gorenstein, $G$ is not a subgroup of $SL(d,k)$, thus $d$ is an odd number
(see \cite{Wata74}). 
\end{proof}


\begin{exam}\label{Fsig_Segre}
If $A$ is a Gorenstein toric ring, then $\fsig(A)=\frac{1}{2}$ does not imply the conditions (2) and (3) in Theorem~\ref{ToricMain}. Namely, consider the Segre product $P_n\coloneqq k[x_1,y_1]\#\cdots\# k[x_n,y_n]$ of $n$ polynomial rings with two variables, 
which is a Gorenstein toric ring in dimension $n+1$. 
Then, by \cite[Proposition~6.1]{HigNak2}, one can compute $\fsig(P_n)=\frac{2}{n+1}$. 
For example, $\fsig(P_3)=\frac{1}{2}$ but the FFRT system of $P_3$ consists of $7$ conic divisorial ideals (see \cite[Example~2.6]{HigNak}). 
\end{exam}

This example shows that  it is difficult to bound F-signature of a Gorenstein ring $A$  using 
the number of modules in the FFRT system. 
For this reason, we give an observation regarding Gorenstein toric rings whose the FFRT system consists of three modules. 

\begin{prop}
\label{ToricFFRT3}
Let $A$ be a toric ring as in Remark~\ref{toric_remark}.  We assume that $A$ is Gorenstein. 
Then, the following conditions are equivalent. 
\begin{enumerate}
\item The FFRT system of $A$ is $\{A, M_1, M_2\}$,
\item $A$ is isomorphic to one of the following rings: 
\begin{enumerate}
\item the invariant subring $k[x_1,\ldots,x_d]^G$ where 
$G=\langle\text{\rm diag}(\underbrace{\xi,\ldots,\xi}_{m},\underbrace{\xi^2,\ldots,\xi^2}_{m})\rangle$ with $d=2m$ and $\xi$ is a primitive cubic root of unity, 
in which case $\fsig(A)=\frac{1}{3}$,  
\item the Segre product $k[x_1,y_1]\# k[x_2,y_2]=k[x_1x_2,x_1y_2,y_1x_2,y_1y_2]$ of two polynomial rings, in which case $\fsig(A)=\frac{2}{3}$. 
\end{enumerate}
\end{enumerate}
\end{prop}

\begin{proof}
We first show $(1)\Longrightarrow(2)$. 
By \cite[Proposition~2.1]{HigNak2}, we have that $M_1\cong M_2^*$. 
Since $\frp_i$ is conic, each $\frp_i$ is isomorphic to either $M_1$ or $M_1^*$. 
Moreover, since $[\omega_A]=[\frp_1]+\cdots+[\frp_n]=0$, we see that $n$ is an even number and we may assume that 
\[
[\frp_1]=\cdots=[\frp_m]=-[\frp_{m+1}]=\cdots=-[\frp_n]
\]
where $n=2m$. 
Then, we see that $\Cl(A)$ is generated by $[\frp_1]$, and we have two cases depending on whether it is torsion. 
 

\begin{itemize}
\item If $\Cl(A)\cong \bbZ/r\bbZ$, then, by an argument similar to the proof of Proposition~\ref{ToricFFRT}, $G\cong \bbZ/3\bbZ$. 
For a generator $g$ of $G$, we can set $g([\frp_1])=\xi$ where $\xi$ is a primitive cubic root of unity. 
In this case, the action of $G$ in $S$ can be described as 
\[
\begin{cases}
g\cdot x_i=\xi x_i&(i=1,\ldots,m),\\
g\cdot x_i=\xi^{-1}x_i=\xi^2 x_i&(i=m+1,\ldots,n), 
\end{cases}
\]
and we have the case (a). 
The F-signature of $A$ can be obtained from \cite[Theorem 4.2]{WataYos04}. 

\medskip

\item
If $\Cl(A)\cong \bbZ$,  then we see that $G\cong k^\times$ and if $g([\frp_1])=\zeta\in k^\times$ for a generator $g$ of $G$, then $g(-[\frp_1])=\zeta^{-1}$. 
Thus, the action of $G$ on $B$ can be described as 
\[
\begin{cases}
g\cdot x_i=\zeta x_i&(i=1,\ldots,m),\\
g\cdot x_i=\zeta^{-1}x_i&(i=m+1,\ldots,n). 
\end{cases}
\]
In this case, we have that 
\[
A\cong k[x_1,\ldots,x_m]\#k[x_{m+1},\ldots,x_n]=k[x_ix_j\mid i=1,\ldots,m,\, j=m+1,\ldots,n]. 
\]
This Segre product of two polynomial rings can be considered as a Hibi ring \cite{Hib}, and the conic classes in Hibi rings are characterized in \cite{HigNak}. 
By \cite[Theorem~2.4 and Example~2.6]{HigNak}, we see that the Segre products of two polynomial rings 
that satisfy the condition $(1)$ are only the one with $m=2$. 
Thus, we have that $\fsig(A) =\frac{2}{3}$ by Example~\ref{Fsig_Segre} (the case of $n=2$), 
or we easily see that $A\cong k[x,y,z,w]/(xw - yz )$, 
in which case $\fsig(A) = \rme(A) - \ehk(A) = 2 - \frac{4}{3}= \frac{2}{3}$. 
\end{itemize}

$(2)\Longrightarrow(1)$ is well known, see e.g. \cite[Proposition 3.2]{SmVand97} for the case (a) and \cite[the proof of Theorem~6.1]{TodYas} for (b). 
\end{proof}

\subsection{Maximal toric F-signatures}\label{maximaltoric Sec}

In this subsection, we completely classify the toric rings $A$ for which $s(A) > \frac{1}{2}$. To this purpose, it will be helpful to introduce the following notation.

\begin{defn}  For $w_1,\dots,w_m\in \sfN$, a finite set of vectors in the lattice $\sfN$, define the \emph{dual zonotope} as
\[ \PC{w_1,\dots,w_m} = \{\mathbf{x} \in \sfM_{\bbR} \;|\; \langle \mathbf{x} , w_i \rangle \in [0,1]\}\,. \]
\end{defn}

In this notation, Theorem~\ref{Fsig_toric} says that $s(k[\sigma^{\vee} \cap \sfM]) = \PC{v_1,\dots,v_n}$, where $v_1,\dots,v_n$ are the minimal generators of $\sigma$. 

Our techniques will be based on volumes of slices of the unit cube, as in Section~\ref{section_lowerHK}. First, we interpret certain volumes as Eulerian numbers.

\begin{lemma} The volume of the portion of the unit $d$-cube where the sum of the coordinates lies between $k$ and $k+1$ is $\displaystyle \frac{A(d,k)}{d!}$, where $A(d,k)$ denotes the Eulerian number with parameters $d$ and $k$.
\end{lemma}
\begin{proof} The following argument is due to Stanley \cite{Sta}.
The hyperplanes $x_i=x_j$ cut the interior of the unit cube into $d!$ simplices of equal volume. Each can be characterized as the set of points $\Delta_{\sigma}$ where $0<x_{\sigma(1)}<x_{\sigma(2)}<\dots<x_{\sigma(d)}<1$ for some $\sigma\in\mathcal{S}_d$, giving a natural bijection between the simplices and $\mathcal{S}_d$. Define a map
\[
\phi(x_1,\dots,x_d)_i= 
\begin{cases} 
x_{i+1}-x_i & \textrm{ if } x_i<x_{i+1}\textrm{ and }i\not=d \\
1+x_{i+1}-x_i & \textrm{ if } x_i>x_{i+1}\textrm{ and }i\not=d \\
1-x_n & \textrm{ if }i=d \,. \\
\end{cases}
\]
Note that $\phi$ maps into the unit cube, and that $\phi|_{\Delta_{\sigma}}$ is affine with determinant $\pm1$. Further, if $(x_1,\dots,x_d)\in{\Delta_{\sigma}}$, then \[k\leq \phi(x_1,\dots,x_d)_1+\dots+\phi(x_1,\dots,x_d)_d\leq k+1\,,\] where $k$ is the number of {descents} of $\sigma$. Additionally, the map
\[\psi(x_1,\dots,x_d)_i=\lceil{x_i+\dots +x_n}\rceil-({x_i+\dots +x_n})\]
provides an inverse for $\phi$ on its image.
\end{proof}

\begin{lemma}\label{euler} For the Eulerian numbers $A(d,k)$, $\frac{A(d,k)}{d!} > \frac{1}{2}$ if and only if $(d,k)=(1,0)$, $(3,1)$, or $(5,2)$.
\end{lemma}

\begin{proof}
By symmetry, it is clear that $\frac{A(d,k)}{d!} < \frac{1}{2}$ for an even integer $d$. Let $k\geq 3$; we will show that $\displaystyle \frac{A(2k+1,k+j)}{(2k+1)!} < \frac{1}{2}$ by induction. The values can explicitly checked for $k=3$. By twice applying the relation 
\[A(n,m)=(n-m)\ A(n-1,m-1)+(m+1)\ A(n-1,m)\]
one obtains the equality
\begin{align*}
A(2k+1,k+j)=&(k-j+1)!\ A(2k-1,k-j-2)+(k+j+1)!\ A(2k-1,k-j)\\
            &+2(k^2+k-j^2)\ A(2k-1,k-j-1)\,.
\end{align*}
By induction, this is less than $(2k-1)!\ (k^2+\frac{3}{2} k+ \frac{1}{2}),$
which, for $k\geq 3$, is less than $\frac{1}{2}(2k+1)!$ as required.
\end{proof}

\begin{thm}\label{maximaltoric} The toric rings with F-signature greater than $1/2$ are, up to isomorphism, as follows:
\begin{itemize}
\item For a polynomial ring $A$, we have $s(A)=1$.
\item For the coordinate ring $A$ of the Segre product $\bbP^1 \# \bbP^1$, we have $s(A)=\frac{2}{3}$.
\item For the coordinate ring $A$ of the Segre product $\bbP^2 \# \bbP^2$, we have $s(A)=\frac{11}{20}$.
\end{itemize}
\end{thm}

\begin{proof} We use the notation of Subsection~\ref{toric_preliminaries}. We can write
\begin{equation}\label{eq:maxdet} \PC{v_1,\dots,v_n} = \bigcap_{\{j_1, \ldots, j_d\}\subset \{1,\dots,n\}} \PC{v_{j_1},\dots,v_{j_d}} \,.
\end{equation}
By the Jacobian formula,
\[ \operatorname{vol}(\PC{v_{j_1},\dots,v_{j_d}}) = \left|\frac{1}{\det{[v_{j_1},\dots,v_{j_d}]}}\right| \,.\]
Thus, if $s(R)>\frac{1}{2}$, $|\det{[v_{j_1},\dots,v_{j_d}]}|=1$ for any $\{j_1, \ldots, j_d\}\subset \{1,\dots,n\}$. Assume for now that this is the case. If $n=d$, then $\{v_1,\dots,v_n\}$ is a basis for $\sfM$, so $A$ is isomorphic to a polynomial ring.

Consider the case where $n=d+1$. The vectors $\{v_1,\dots,v_d\}$ form a basis for $\sfN$.  Let $x_1,\dots,x_d$ be coordinates of $\sfM$ forming a dual basis to $v_1,\dots,v_d$. We may thus compute the volume of $ \PC{v_1,\dots,v_n}$ in these coordinates. Since $|\det{[v_1,\dots,\hat{v_i},\dots,v_d,v_{d+1}]}|=1$, the $i^{th}$ coordinate of $v_{d+1}$ is $\pm 1$. That is, in these coordinates,
\[ [v_1,\dots,v_{d+1}]=\left[
\begin{array}{ccccc}
1 & 0 & \hdots & 0 & \pm 1 \\
0 & 1 &  \hdots      & 0 & \pm 1 \\
\vdots & \vdots & \ddots & \vdots & \vdots \\
0 & 0 & \hdots & 1 & \pm 1 \\
\end{array}
\right]\,.
\]

Note that if $d\leq 2$, then our generating set for $\sigma$ is not minimal, so we may assume that $d\geq 3$. Renumber the coordinates so that in the matrix above, $\langle x_i, v_{d+1}\rangle=+1$ for $i\leq k$ and $\langle x_i, v_{d+1}\rangle=-1$ for $i>k$. Then $\PC{v_1,\dots,v_{d+1}}$ is the subset of the unit $d$-cube where \[0\leq x_1+\cdots+x_k-x_{k+1}-\cdots-x_d\leq 1\,.\] Using $x_j\mapsto 1-x_j$ symmetry of the cube, we have 
\[ \operatorname{vol}(\PC{v_1,\dots,v_{d+1}}) = \operatorname{vol}(\{(x_1,\dots,x_d)\in[0,1]^d \;|\; k\leq \sum_{i=1}^d x_i \leq k+1 \,\}) \,.\]
By Lemma~\ref{euler}, we see that the volume $s(A)$ is greater than $\frac{1}{2}$ only if $d=3$ and $k=1$ or $d=5$ and $k=2$. These correspond to the coordinate rings of $\bbP^1 \# \bbP^1$ and $\bbP^2 \# \bbP^2$. 

We now consider what happens if $n\geq d+2$.
If $\operatorname{vol}(\PC{v_1,\dots,v_n})>1/2$, with $d=3$, then 
\[ [v_1,\dots,v_4]=\left[
\begin{array}{cccc}
1 & 0 & 0 &  1 \\
0 & 1 & 0 &  1 \\
0 & 0 & 1 &  -1 \\
\end{array}
\right]\]
in the dual basis to $v_1, v_2, v_3$. If $n\geq 5$, the same arguments show that if $\operatorname{vol}(\PC{v_1,\dots,v_n})>\frac{1}{2}$, in a particular basis we have 
\[ [v_1,\dots,v_5]=\left[
\begin{array}{ccccc}
1 & 0 & 0 & 1 & 1\\
0 & 1 & 0 & 1 & -1\\
0 & 0 & 1 & -1 & 1\\
\end{array}
\right]
\]
where one computes $\operatorname{vol}(\PC{v_1,\dots,v_5})=1/3$. Thus, the F-signature cannot be greater than or equal to $1/2$ in this case.
Now consider when $d=5$. If $\operatorname{vol}(\PC{v_1,\dots,v_n})>\frac{1}{2}$, 
\[ [v_1,\dots,v_6]=\left[
\begin{array}{cccccc}
1 & 0 & 0 & 0 & 0 &  1 \\
0 & 1 & 0 & 0 & 0 &  1 \\
0 & 0 & 1 & 0 & 0 &  1 \\
0 & 0 & 0 & 1 & 0 & -1 \\
0 & 0 & 0 & 0 & 1 & -1 \\
\end{array}
\right] \]
in the dual basis to $v_1, v_2, v_3, v_4, v_5$. If $n\geq 7$, a case-by-case analysis similar to above shows that the volume $\operatorname{vol}(\PC{v_1,\dots,v_n})<\frac{1}{2}$.
\end{proof}

If $A$ is not a toric ring, we have a family of $3$-dimensional Gorenstein rings whose F-signature are greater than $\frac{1}{2}$. 

\begin{exam}
\label{ex_cDV}
Let $c>2$ be an integer. 
Put $A:=k[[x,y,z,w]]/(x^2+y^2+z^2+w^c)$, where 
$k$ is an algebraically closed field of characteristic $p >c$. 
Then, $A$ is a normal hypersurface of $\dim A=3$, and 
\[
\dfrac{1}{2} < \fsig(A) < \dfrac{2}{3}. 
\] 
In fact, since $\rme(A)=2$ we have that $\fsig(A)=2-\ehk(A)$, thus this follows from \cite[Corollary~3.11]{WataYos05}. 
\end{exam}

\subsection*{Acknowledgements} 

This work started during the third author's stay at the Nihon University and he thanks them for hospitality. 
The second author would like to thank Kazuhiko Kurano, Akihiro Higashitani, and Kohsuke Shibata for valuable discussions.
The authors also want to thank the anonymous referee for the insightful comments that improved the quality and readability of our manuscript.

The first author is supported by by NSF CAREER Award DMS-2044833.
The second author is supported by World Premier International Research Center Initiative (WPI initiative), MEXT, Japan,  JSPS Grant-in-Aid for Young Scientists (B) 17K14159, and 
JSPS Grant-in-Aid for Early-Career Scientists 20K14279.
At various times during the preparation of this manuscript, the third author had support of a Starting Grant 2020-03970 by the Swedish Research Council and of a fellowship from ``la Caixa'' Foundation (ID 100010434) and from the European Union’s Horizon 2020 research and innovation programme under the Marie Skłodowska-Curie grant agreement No 847648. The fellowship code is ``LCF/BQ/PI21/11830033''.
The fifth author was partially supported by JSPS Grant-in-Aid for Scientific Research (C) Grant Numbers 19K03430. 

The content of subsection~\ref{maximaltoric Sec} originally appeared in the first author's Ph.D. thesis \cite{JackThesis}.


\providecommand{\bysame}{\leavevmode\hbox to3em{\hrulefill}\thinspace}
\providecommand{\MR}{\relax\ifhmode\unskip\space\fi MR }
\providecommand{\MRhref}[2]{%
  \href{http://www.ams.org/mathscinet-getitem?mr=#1}{#2}
}
\providecommand{\href}[2]{#2}

\end{document}